\documentclass[12pt]{amsart}

\usepackage{color}

\usepackage{amsmath,amssymb,setspace,nicefrac, yhmath, amscd,eucal,enumitem}
\usepackage[active]{srcltx}
\usepackage[colorlinks, linkcolor=red, citecolor=blue, urlcolor=blue, hypertexnames=true]{hyperref}
\usepackage{amsrefs, mathrsfs}

\allowdisplaybreaks

\newcommand{\R}{{\mathbb R}}

\newcommand{\Z}{{\mathbb Z}}
\newcommand{\N}{{\mathbb N}}
\newcommand{\F}{{\mathbb F}}

\newcommand{\K}{{\mathbb K}}

\newcommand{\cC}{{\mathcal C}}

\newcommand{\cP}{{\mathcal P}}
\newcommand{\cN}{{\mathcal N}}

\newcommand{\cR}{{\mathcal R}}

\newcommand{\cH}{{\mathcal H}}

\newcommand{\Aut}{{\rm Aut}}

\newcommand{\ldense}{{\rm ldense}}
\newcommand{\udense}{{\rm udense}}
\newcommand{\dense}{{\rm dense}}
\newcommand{\inspec}{{\rm inspec}}

\newcommand{\diam}{{\rm diam}}
\newcommand{\hspec}{{\rm hspec}}

\newtheorem{thm}{Theorem}[section]
\newtheorem{cor}[thm]{Corollary}
\newtheorem{lem}[thm]{Lemma}

\newtheorem{definition}[thm]{Definition}

\newtheorem{remark}[thm]{Remark}
\newtheorem{proposition}[thm]{Proposition}

\theoremstyle{definition}

\begin{document}
    \title[Riordan groups over finite fields]{Algebraic and topological properties of Riordan groups over finite fields}

    \author{Gi-Sang Cheon }
    \author{Nhan-Phu Chung}
    \author{Minh-Nhat Phung}
   \address{Department of Mathematics, Sungkyunkwan University, 2066 Seobu-ro, Jangan-gu, Suwon-si, Gyeonggi-do 16419, Korea. Tel: +82 031-299-4819}
        \address{Applied Algebra and Optimization Research Center, Sungkyunkwan University, Korea.}
         \address{Institute of Applied Mathematics, University of Economics Ho Chi Minh city, Vietnam. }
    \email{Gi-Sang Cheon, gscheon@skku.edu}
    \email{Nhan-Phu Chung, phucn@ueh.edu.vn;phuchung82@gmail.com}
    \email{Minh-Nhat Phung, pmnt1114@gmail.com}


    \date{\today}
    \maketitle
    \begin{abstract}
        In this paper, we investigate algebraic and topological properties of the Riordan groups over finite fields. These groups provide a new class of  topologically finitely generated, abstract, pro-p groups with finite width. We also introduce, characterize index-subgroups of our Riordan groups, and finally we show exactly the range of Hausdorff dimensions of these groups. The latter results are analogous to the work of Barnea and Klopsch for the Nottingham groups.
    \end{abstract}
    \section{Introduction}
    The Riordan group over the field of real or complex numbers was introduced in 1991 by Shapiro and his collaborators \cite{Shapiro} in the framework of enumerative combinatorics. After that, it has been investigated by many authors with applications in combinatorics, computer science, group theory, matrix theory, number theory, Lie groups and Lie algebras, orthogonal polynomials, graph theory, Heisenberg-Weyl algebra. A non-exhaustive list of works related to Riordan groups is given by \cite{MR4119405,MR3464075,MR2451085,MR3771671,MS,MSV,Sprugnoli94, Sprugnoli11} and references there in. Recently, in 2017 via the inverse limit approaches of Riordan groups initiated in \cite{LMMPS}, the authors established an infinite-dimensional Fr\'{e}chet Lie group and the corresponding Lie algebra on the Riordan group \cite{CLMPS}.

    In this article, we study Riordan groups $\cR(\K)$ over general unital commutative rings $\K$ and more specially for $\K=\F_q$ finite fields. To our knowledge, this is the first paper investigating properties of these Riordan groups. Instead of using the representation of Riordan groups via Riordan arrays which are infinite lower triangular matrices, we construct them as semi-direct products of $\cH(\K)$, the group of formal series with the free coefficient is 1, and $\cN(\K)$, the substitution group of formal power series. The group $\cN(\K)$ was first introduced in 1954 by Jennings \cite{Jennings}, and was brought attention later to the group theory community by Johnson \cite{Johnson} and his Ph.D student York \cite{York1, York}. It is known under the name of the Nottingham group now. Since then, the group $\cN(\F_p)$, where $\F_p$ is the finite field with a prime number $p$, has been extensively studied in the literature. It has many remarkable properties and plays an important role in the theory of profinite groups \cite{Babenko13,BS,Index,Camina,dSF,Ershov, Ershov1, Klopsch,Sz}. As our Riordan group $\cR(\F_p)$ contains the Nottingham group $\cN(\F_p)$ as a subgroup, we would like to investigate which properties of $\cN(\F_p)$ still hold for $\cR(\F_p)$.

    It was proved that the Nottingham groups $\cN(\F_p)$ and $\cN(\Z)$ are topologically finitely generated \cite{Camina,BB08}. On the other hand, the group $\cN(\F_p)$ has finite width \cite{LM}. Our first main result is to show that $\cR(\F_p)$ and $\cR(\Z)$ are topologically finitely generated, and $\cR(\F_p)$ has finite width.
    \begin{thm}
        \label{T-main1}
        Let $\K$ be a finite commutative unital ring, $p\geq 2$ be a prime number and $\F_q$ be a finite field with characteristic $p$. Then
        \begin{enumerate}
        \item the Riordan group $\cR(\K)$ is profinite;
        \item the Riordan group $\cR(\F_q)$ is a pro-p-group with finite width;
            \item the Riordan groups $ \cR(\F_p) $  and $ \cR(\Z) $ are topologically finitely generated. Furthermore, for every topologically generating set $S_N$ of the Nottingham group $ \cN(\F_p)$ (respectively $\cN(\Z))$, the group $ \cR(\F_p) $ (respectively $ \cR(\Z) $) is generated by $$ \{(1+x,x), (1,g):g\in S_N\} .$$
            \end{enumerate}
    \end{thm}
    The definitions of $\cR(\K)$, topologically finitely generating sets and groups with finite width will be presented in Section 2.

    On the other hand, the study of Hausdorff dimension in profinite groups has been initiated by Abercrombie \cite{Ab}. Given a filtration $\{G_n\}_n$ of an infinite profinite group $G$, we can define an induced invariant metric on $G$ and then we can compute the Hausdorff dimension of a closed subgroup $H$ of $G$ with respect to this metric via the following formula
    $$\dim_H(H)=\liminf_{n\to \infty}\frac{\log |HG_n/G_n|}{\log |G/G_n|}.$$
    After that, in \cite{BS}, given an infinite profinite group $G$, Barnea and Shalev investigated the dimension spectrum of $G$
    $$\hspec(G):=\{\dim_H(H):H \mbox{ is a closed subgroup of } G\}.$$
    Using Lie algebra ideas they showed that if $p\geq 5$ is a prime number then
    $$\bigg\{\frac{1}{s}:s\in \N\bigg\}\subset \hspec(\cN(\F_p))\subset [0,\frac{3}{p}]\cup\bigg\{\frac{1}{s}:s\in \N\bigg\}.$$
    Later, Barnea and Klopsch introduced index-subgroups, a large class of closed subgroups of $\cN(\F_p)$ \cite{Index}. Then they partially extended the above result of \cite{BS} as they proved that for $p>2$
    $$\inspec(\cN(\F_p))=[0,\frac{1}{p}]\cup\bigg\{\frac{1}{p}+\frac{1}{p^r}:r\in \N\bigg\}\cup\bigg\{\frac{1}{s}\bigg\},$$
    where $$\inspec(\cN(\F_p)):=\{\dim_H(H):H \mbox{ is an index-subgroup of } \cN(\F_p)\}\subset \hspec(\cN(\F_p)).$$
    Inspired by the work of \cite{Index}, we introduce index-subgroups of the Riordan group $\cR(\F_p)$. In our second main result of the article, we characterize these index-subgroups and given a standard filtration of $\cR(\F_p)$ we describe exactly the set $$\inspec(\cR(\F_p)):=\{\dim_H(H):H \mbox{ is an index-subgroup of } \cR(\F_p)\}\subset \hspec(\cR(\F_p)).$$
    \begin{thm}
        \label{T-main 2}
        Let $p>2$ be a prime number and given the filtration $\{ \cH^n(\F_p)\rtimes \cN^n(\F_p) \}_n$ of $\cR(\F_p)$, we have \begin{align*}
            \inspec(\cR(\F_p))=[0,\dfrac{1}{p}]&\cup\bigg\{\dfrac{1}{p}+\dfrac{1}{2p^r}|r\in\N\bigg\}\cup\bigg\{\dfrac{1}{2}+\dfrac{1}{2p}+\dfrac{1}{2p^r}|r\in\N\bigg\}\\&\cup\bigcup_{s<p}[\dfrac{1}{2s},\dfrac{1}{2s}(1+\dfrac{1}{p})]\cup\bigg\{\dfrac{1}{2sp^r}+\dfrac{1}{2su}|s,u\in\N,r\in\N\cup\{0\}\bigg\}.
        \end{align*}
    \end{thm}
    In contrast to the Nottingham group $\cN(\F_p)$, Hausdorff dimensions of the Riordan group $\cR(\F_p)$ do depend on the choice of the filtrations.

    Our paper is organized as follows. In Section 2 we define our Riordan group $\cR(\K)$ as a semi-direct product of $\cH(\K)$ and $\cN(\K)$, and then we present it as an inverse limit of topological groups. After that we show that both $\cR(\F_p)$ and $\cR(\Z)$ are topologically finitely generated, calculate the lower central series of $\cR(\F_q)$ and then prove Theorem \ref{T-main1}. Finally, in Section 3 we first review the Hausdorff dimension on metric spaces arising from profinite groups, introduce and characterize the index-subgroups of $\cR(\F_p)$, and finally we present a proof of Theorem \ref{T-main 2}.

    \textbf{Acknowledgements:} This work was partially supported by Science Research Center Program through the National Research Foundation of Korea (NRF) Grant funded by the Korean Government (MSIP) (NRF-2016R1A5A1008055). G.-S. Cheon was partially supported by the NRF-2019R1A2C1007518. N.-P. Chung and M.-N. Phung were partially supported by the NRF-2019R1C1C1007107. We thank Minho Song for pointing out \cite[Proposition 3.3]{GG} to us.
    \section{Riordan groups and their properties}
    In this note, we put $\N_0=\N\cup\{0\}$ and $p$ is always a prime number.

    For every unital commutative ring  $\K$, we denote by $\K[[x]]$ the set of all formal series with coefficients in $\K$, i.e.
    $$\K[[x]]:=\big\{\sum_{i=0}^\infty a_ix^i:a_i\in \K, \mbox{ for every } i\in \N_0\big\}.$$
    Under the usual multiplication, we know that an element $g=\sum_{i=0}^\infty a_ix^i\in \K[[x]]$ is invertible if and only if $a_0$ is invertible in $\K$.

    We denote $\cH(\K):=\{1+a_1x+\dots\in \K[[x]]\}$ and endow it with the usual multiplication operation $\cdot:\cH(K)\times \cH(K)\to \cH(K)$ then $(\cH(\K),\cdot) $ is a group. We consider $\cN(\K)$ the set of all $g=x+a_2x^2+\dots+a_nx^n+\dots \in\K[[x]]$, and define the substitution operation $\circ: \cN(\K)\times \cN(\K)\to \cN(\K)$ by
    $$h\circ g:=g(h)=h+\sum_{i=2}^\infty a_ih^i,$$
    for every $g=x+\sum_{i=2}^\infty a_ix^i, h=x+\sum_{i=2}^\infty b_ix^i\in \cN(\K)$.
    It is well known that $(\cN(\K),\circ)$ is a group \cite[Theorem 1.1]{Jennings}. It is called the Nottingham group which has been studied intensively.

    Now we define our Riordan group as a semidirect product of $\cH(\K)$ and $\cN(\K)$ as follows. We denote by $\Aut(\cH(\K))$ the group of all automorphisms of $\cH(\K)$.
    For every $h= 1+\sum_{i=1}^\infty b_ix^i\in \cH(\K)$ and $g\in \cN(\K)$, the substitution $g\circ h=1+\sum_{i=1}^\infty b_ig^i$ is still in $\cH(\K)$.
    Let $\varphi: \cN(\K)\to \Aut(\cH(\K))$ be the map defined by $\varphi(g)(h)=g\circ h$ for every $h\in \cH(\K), g\in \cN(\K)$ then it is clear that $\varphi$ is well defined and indeed it is a homomorphism.
    Let $\cR(\K)$ be the semi-direct product $\cH(\K)\rtimes_\varphi \cN(\K)$. The multiplication operation in $\cH(\K)\rtimes_\varphi \cN(\K)$ is defined by
    $$(h_1,g_1)\cdot (h_2,g_2):=(h_1\varphi(g_1)h_2,g_1\circ g_2)=(h_1h_2(g_1),g_1\circ g_2),$$ for every $(h_1,g_1),(h_2,g_2)\in \cH(\K)\times \cN(\K)$.
    For every $(h,g)\in \cH(\K)\rtimes_\varphi \cN(\K)$, its inverse is
  $(h,g)^{-1}=(\varphi(\bar{g})(h^{-1}),\bar{g})=(h^{-1}(\bar{g}),\bar{g})$, where $hh^{-1}=1$ and $g\circ {\bar g}={\bar g}\circ g=x$.


    We start with the following elementary lemma which will be used several times later.
    \begin{lem} Let $h=1+\sum_{i=n}^\infty a_ix^i\in \cH(\K)$ and $g=x+\sum_{j=m}^\infty b_jx^j\in \cN(\K)$. Then
        \begin{enumerate}
            \label{L-simple computations}
            \item the invertible element $h^{-1}$ in $\cH(\K)$ is $1+\sum_{i=n}^\infty (p_i-a_i)x^i$, where $p_n=0$ and $p_k$ is a polynomial depending only on $a_n,\cdots, a_{k-1}$ for every $k>n$;
            \item For every $k\geq n$, the coefficient at degree $k$ of $h(g)$ is: \begin{enumerate}[label=(\roman*)]
                \item $ a_k $ if $ n\leq k<m+n-1 $,
                \item $ na_nb_m+a_{m+n-1} $ if $ k=m+n-1 $,
                \item $ na_nb_{k-n+1}+q_{k} $, where $ q_{k} $ is a polynomial depending only on $a_n,\cdots, a_{k}$ and $b_m,\cdots, b_{k-n}$ if $ k> m+n-1 $;
            \end{enumerate}
            \item For every $k\geq n$, the coefficient at degree $k$ of $h(g)h^{-1}$ is \begin{enumerate}[label=(\roman*)]
                \item $ 0 $ if $ n\leq k<m+n-1 $,
                \item $ na_nb_m $ if $ k=m+n-1 $,
                \item $ na_nb_{k-n+1}+r_{k} $, where $ r_{k} $ is a polynomial depending only on $a_n,\cdots, a_{k}$ and $b_m,\cdots, b_{k-n}$ if $ k> m+n-1 $.
            \end{enumerate}
        \end{enumerate}
    \end{lem}
    \begin{proof}
        (1) We write $h^{-1}$ in the form $1+c_1x+c_2x^2+\cdots c_{n-1}x^{n-1}+\sum_{i=n}^\infty (p_i-a_i)x^i$.
            Because $ hh^{-1}=1 $ we get that $c_1=\cdots=c_{n-1}=0$, $p_n=0$ and  for every $k\geq n+1$,  $p_n $ must satisfy the recurrence \begin{align*}
                a_k+p_k-a_k+\sum_{i=n}^{k-1}a_i(p_{k-i}-a_{k-i})=0.
            \end{align*}
            (2)
            Reminding that $\left( \sum_{j=1}^{\infty}x_j \right)^{n}=\sum_{\mathbf{t}}c_{\mathbf{t}}\prod_{j=1}^{\infty}x_j^{t_j}$, where $\mathbf{t}=(t_1,t_2,\dots)\subset\N\cup\{0\}$ satisfies $\sum_{j=1}^{\infty}t_j=n$, and the coefficient $c_{\mathbf{t}}=\frac{n!}{\prod_{j=1}^{\infty}t_j!}$.
			Using the formula for $n=i, x_1=1$ and $x_{j-m+2}=b_{j}x^{j-1},j\ge m$ to expand $ h(g) $ we get that \begin{align*}
                h(g)=&1+\sum_{i=n}^\infty a_{i}x^{i}(1+\sum_{j=m}^\infty b_jx^{j-1})^{i}\\
                =&1+\sum_{i=n}^{\infty}a_i\left(\sum_{(t_j)^i_{j\geq m}}\dfrac{i!}{(i-S_1({(t_j)^i_{j\geq m}}))!}\prod_{j}\dfrac{b_{j}^{t_j}}{t_j!}\right)x^{i+S_2({(t_j)^i_{j\geq m}})},
            \end{align*} where $ (t_j)^i_{j\geq m} $ is taken over all sequences $ (t_j)_{j\geq m}\subset\N\cup\{0\} $ such that $ S_1:=\sum_{j\geq m}t_j\leq i $, and where $ S_2((t_j)_{j\geq m}):=\sum_{j\geq m}(j-1)t_j $. In order to find coefficient for $ x^k $, we consider $ i+S_2({(t_j)^i_{j\geq m}})=k $.
            \begin{enumerate}[label=(\roman*)]
                \item For $ k<m+n-1 $ we only have the case $ i=k,S_2({(t_j)^i_{j\geq m}})=0 $ so $ S_1=0 $ and $ \prod_{j}\dfrac{b_{j}^{t_j}}{t_j!}=1 $. Thus, the coefficient is $ a_k $.
                \item For $ k=m+n-1 $ we have $ i=n,S_2({(t_j)^i_{j\geq m}})=m-1 $ or $ i=m+n-1,S_2({(t_j)^i_{j\geq m}})=0 $. If $ i=n,S_2({(t_j)^i_{j\geq m}})=m-1 $ we have $ na_nb_m $ and if $ i=m+n-1,S_2({(t_j)^i_{j\geq m}})=0 $ we have $ a_{m+n-1} $.
                \item If $ t_j>0 $ for $ j>k-n+1 $ then $ S_2({(t_j)^i_{j\geq m}})>k-n $. So, we only consider $ t_j=0 $ for $ j>k-n+1 $. For $ j=k-n+1 $ if $ t_j>1 $ then $ S_2({(t_j)^i_{j\geq m}})\geq 2(k-n)>k-n $. So $ t_{k-n+1}=0 $ or $ t_{k-n+1}=1 $. The case $ t_{k-n+1}=1 $ leads to $ S_2({(t_j)^i_{j\geq m}})=k-n $, this means that $ t_{k-n+1} $ is the only non-zero number in the sequence and $ i=n $. Thus, $ S_1=1 $ and $ \prod_{j}\dfrac{b_{j}^{t_j}}{t_j!}=b_{k-n+1} $. Hence, we get the part $ na_nb_{k-n+1} $ in the coefficient. For $ q_k $, we have $ t_j=0 $ for $ j\geq k-n+1 $ so we get the desired result.
            \end{enumerate}
            (3) From (1) we get that the coefficient at degree $ k $ of $ h(g)h^{-1} $ is $ \sum_{i=0}^{k}c_i(p_{k-i}-a_{k-i}) $, where $ c_i $ is the coefficient at degree $ i $ of $ h(g) $.
            \begin{enumerate}[label=(\roman*)]
                \item For $ n\leq k<m+n-1 $ we have $ c_i=0 $ if $ 0<i<n $, and $ c_i=a_i $ if $ n\leq i\leq k $.  \begin{align*}
                    \sum_{i=0}^{k}c_i(p_{k-i}-a_{k-i})=a_k+(p_k-a_k)+\sum_{n}^{k-1}a_i(p_{k-i}-a_{k-i})=0.
                \end{align*}
                \item Because $ c_{m+n-1}=na_nb_m+a_{m+n-1} $ so \begin{align*}
                    \sum_{i=0}^{k}c_i(p_{k-i}-a_{k-i})=na_nb_m+a_k+(p_k-a_k)+\sum_{n}^{k-1}a_i(p_{k-i}-a_{k-i})=na_nb_m.
                \end{align*}
                \item For $ k>m+n-1 $, we have \begin{align*}
                    \sum_{i=0}^{k}c_i(p_{k-i}-a_{k-i})=na_nb_{k-n+1}+q_k+(p_k-a_k)+\sum_{n}^{k-1}c_i(p_{k-i}-a_{k-i}).
                \end{align*}
                We see that the property $ r_k=q_k+(p_k-a_k)+\sum_{n}^{k-1}c_i(p_{k-i}-a_{k-i}) $ is a polynomial depending only on $ a_n,\ldots,a_k $ and $ b_m,\ldots,b_{k-n} $ follows from (1) and (2).
            \end{enumerate}

    \end{proof}
    \subsection{Profiniteness of $\cR(\K)$}
    Before proving properties of the Riordan group $\cR(\K)$, let us review the definitions of inverse limit groups and profinite groups. For more details of profinite groups, see \cite{RZ,Wilson}. A \textit{directed set} is a partially order set $(I,\leq)$ such that for all $a,b\in I$, there exists an element $c\in I$ such that $a\leq c$ and $b\leq c$. Let $I$ be a directed set. An \textit{inverse system} $(X_i,\pi_{ij})_{i,j\in I, i\leq j}$ of topological spaces consists of a family $(X_i)_{i\in I}$ of topological spaces and a family of continuous maps $(\pi_{ij}: X_j\to X_i)_{i,j\in I, i\leq j}$ such that $\pi_{ii}=Id_{X_i}$, the identity map of $X_i$ for every $i\in I$, and $\pi_{ik}=\pi_{ij}\pi_{jk}$ for every $i\leq j\leq k$. If each $X_i$ is a topological group and each $\pi_{ij}$ is a continuous homomorphism then $(X_i,\pi_{ij})$ is called an inverse system of topological groups. An \textit{inverse limit} of
    an inverse system  $(X_i,\pi_{ij})$ of topological groups is a topological group $X$ with a continuous homomorphisms $\pi_i:X\to X_i$ satisfying $\pi_{ij}\circ\pi_j=\pi_i$ for every $i\leq j$ in $I$ and the following universal property: whenever $(\varphi_i:Y\to X_i)$ is a family of continuous homomorphisms such that $\pi_{ij}\varphi_j=\varphi_i$ for every $i\leq j$ in $I$ there is a unique continuous homomorphism $\varphi:Y\to X$ such that $\pi_i\varphi=\varphi_i$ for every $i\in I$. By the universal property, two inverse limit group of  $(X_i,\pi_{ij})$ will be isomorphic. Let $P=\prod_{i\in I}X_i$ denote the direct product of groups $X_i$ then
    $$X:=\{(x_i)\in P: \pi_{ij}(x_j)=x_i \mbox { for all } i,j\in I \mbox{ such that } i\leq j\}$$ is a subgroup of $P$. Then the group $X$ is the inverse limit group of $(X_i)$ and we denote $X=\varprojlim X_i$. We endow $X$ with the topology induced from the product topology of $\prod_{i\in I}X_i$.

    Let $\cC$ be a class of groups. A group $H\in\cC$ is called a $\cC$-group and an inverse limit group $G$ of $\cC$-groups is called a pro-$\cC$-group. In the special case where $\cC$ is the class of all finite groups endowed with the discrete topology, we call an inverse limit group $G$ of $\cC$-groups is profinite. Furthermore, if $p$ is a prime and $G=\varprojlim G_i$, where $G_i$ is a $p$-group for every $i$, we say that $G$ is a pro-$p$-group. Recall that a finite group is called a $p$-group, where $p$ is a prime, if its order is a power of $p$.

    Now we are ready to prove our first result in this section that the Riordan group $\cR(\K)$ is profinite for every finite commutative ring $\K$.

    For every $n\in \N$, we define
    \begin{align*}
        \cH^n(\K):&=\{h(x)=1+a_nx^n+a_{n+1}x^{n+1}+\cdots\in \cH(\K)\}, n\geq 2,\\
        \cH^1(\K):&=\cH(\K),\\
        \cN^n(\K):&=\{g(x)=x+a_{n+1}x^{n+1}+a_{n+2}x^{n+2}+\cdots\in \cN(\K)\}, n\geq 2,\\
        \cN^1(\K):&=\cN(\K).
    \end{align*}

    For every $m,n\in \N$, we consider the subsets $\cR^{m,n}(\K):=\cH^m(\K)\rtimes_{\varphi_n|_{\cH^{m}(\K)}} \cN^n(\K)$ of $\cR(\K)$, where $\varphi_n$ is the restriction of $\varphi:\cN(\K)\to \Aut(\cH(\K))$ to $\cN^n(\K)$ and $ \varphi_n|_{\cH^{m}(\K)} $ is defined by the restriction of $ \varphi_n(g) $ to $ \cH^m(\K) $ for each $ g\in \cN^n(\K) $. We use the notation $ \cR^n(\K) $ when $ m=n $.
    \begin{lem}\label{Rmn}
        \begin{enumerate}
            \item For every $m,n\in \N$, the map $ \varphi_n|_{\cH^{m}(\K)}:\cN^n(\K)\to\Aut(\cH^{m}(\K)) $ is well defined and therefore so is $ \cR^{m,n}(\K) $.
            \item $ \cR^{m_1,n_1}(\K) $ is a normal subgroup of $ \cR^{m_2,n_2}(\K) $ for $ m_2+n_1\geq m_1\geq m_2$ and $n_1\geq n_2 $.
        \end{enumerate}
    \end{lem}
    \begin{proof}
        (1) We need to show that for any $ g\in \cN^n(\K) $ and $ h\in \cH^m(\K) $ there exists a unique $ h'\in \cH^m(\K) $ such that $ h'\circ g=h $. Since we already had $ \varphi(g)\in\Aut(\cH(\K)) $ so there exists a unique $ h'\in \cH(\K) $ such that $ h'\circ g=h $. We want to show $ h'\in \cH^m(\K) $. Suppose $ h'\in \cH(\K)\setminus \cH^m(\K) $ then there exist $ m'<m$ and $a_{m'}\neq 0 $ such that $ h'=1+a_{m'}x^{m'}+\mathscr{O}(x^{m'+1}) $. Because $ g=x+\mathscr{O}(x^{n+1}) $ so $ h=h'\circ g=1+a_{m'}x^{m'}+\mathscr{O}(x^{m'+1})\notin \cH^m(\K) $, a contradiction.

            (2) For every $ h_i\in \cH^{m_i},g_j\in \cN^{n_j} $ we have
            \begin{align*}
	      (h_2,g_2)^{-1}(h_1,g_1)(h_2,g_2)&=(\varphi(\bar{g_2})(h_2^{-1}),\bar{g_2})(h_1,g_1)(h_2,g_2)\\
	      &=(\varphi(\bar{g_2})(h_2^{-1})\varphi(\bar{g_2})(h_1),\bar{g_2}\circ g_1)(h_2,g_2)\\
	      &=(\varphi(\bar{g_2})[h^{-1}_2h_1\varphi(g_1)(h_2)],\bar{g_2}\circ g_1\circ g_2).
            \end{align*}
            Let $ h=1+ax^m+\mathscr{O}(x^{m+1}) $ and $ g=x+bx^{n+1}+\mathscr{O}(x^{n+2}) $ then applying Lemma \ref{L-simple computations} we have \begin{align*}
                h(g)h^{-1}=(h+abmx^{m+n}+\mathscr{O}(x^{m+n+1}))h^{-1}\in \cH^{m+n}(\K) .
            \end{align*}
	  Hence, $ \varphi(\bar{g_2})[h^{-1}_2h_1\varphi(g_1)(h_2)]\in \cH^{m_1}(\K) $. From \cite[Lemma 2.3]{Babenko13}, we also have $ \bar{g_2}\circ g_1\circ g_2\in \cN^{n_1}(\K) $. The proof is complete.
\end{proof}
    If $\K$ is a finite unital commutative ring then the group $\cR_n(\K):=\cR(\K)/\cR^n(\K) $ is finite for every $n\in \N$. Furthermore, if $p$ is a prime and $\K=\F_q$ a finite field with characteristic $p$ then $\cR_n(\K)$ is a $p$-group for every $n\in \N$. Therefore, we get the following corollary.
    \begin{cor}
    \label{C-profinite}
        For every unital commutative ring $\K$, $ \cR(\K)=\varprojlim\cR_n(\K)$. If $\K$ is finite then $\cR(\K)$ is profinite. Furthermore, if $\K= \F_q$ a finite field with characteristic $p$, where $p$ is a prime then $\cR(\K)$ is a pro-$p$-group.
    \end{cor}
    \begin{remark}
    We have many other representations of $\cR(\K)$ as inverse limits of topological groups, for example $\cR(\K)=\varprojlim \cP_n(\K)$, where $\cP_n(\K):=\cR(\K)/\cH(\K)\rtimes \cN^n(\K)$ for every $n\in \N$. However, we choose our representation $\cR(\K)=\varprojlim\cR_n(\K) $ and its induced inverse limit topology for the remaining of the paper.
    \end{remark}
    From now on, if there is no other statement we assume the action is induced from substitutions. We just denote the semi-direct products by $ \rtimes $ instead of $ \rtimes_\varphi $ or $ \rtimes_{\varphi_n|_{\cH^{m}(\K)}} $.

    For every element $h=\sum_{n=0}^\infty h_nx^n\in \K[[x]]$ and every $m\in \N_0$, we denote $[x^m]h:=h_m\in \K$.
    For each $(h,g)\in \cR(\K)$ we associate the matrix $A(h,g)\in M_{\N_0}(\K)$ with entries $a_{ij}\in \K$ defined by $a_{ij}:=[x^i]h(x)g^j(x)$ for every $i,j\in \N_0$. Then $A(h,g)$ is a lower triangle matrix, i.e. $a_{ij}=0$ for every $j>i$, and furthermore, $a_{ii}=1$ for all $i\in \N_0$. The map $A:\cR(\K)\to GL_{\N_0}(\K), (h,g)\mapsto A(h,g)$ is a homomorphism and indeed it is a monomorphism. We denote the map $ \tilde{\cR}(\K) $ as the image $ A(\cR(\K)) $ and denote $ \tilde{\cR}^n(\K) $ is the subset of $ \tilde{\cR}(\K) $ such that $ a_{ij}=0 $ for $ j<i<n $. We have the following proposition.
    \begin{proposition}
        For every $n\in \N$, $ \tilde{\cR}^n(\K) $ is a normal subgroup of $ \tilde{\cR}(\K) $ and is isomorphic to $ \cH^n(\K)\rtimes \cN^n(\K) $ by the restriction on $ A $. Then, naturally, $ \tilde{\cR}(\K)/\tilde{\cR}^n(\K) $ is isomorphic to $ \cR_n(\K) $.
    \end{proposition}
    \begin{proof}
        First, we show that $ A(\cH^n(\K)\rtimes \cN^n(\K))\subseteq\tilde{\cR}^n(\K) $. Indeed, let $ h=1+h_nx^n+\ldots\in \cH^n(\K) $ and $ g=x+g_nx^{n+1}+\ldots\in \cN^n(\K) $ then we see that \begin{align*}
            [A(h,g)]_{ij}=[x^i]h(x)g^j(x)=&[x^{i-j}](1+h_nx^n+\ldots)(1+g_nx^n+\ldots)^j\\
            =&[x^{i-j}](1+(h_n+jg_n)x^n+(h_{n+1}+\ldots)x^{n+1}+\ldots)\\
            =&0,
        \end{align*} for $ j<i<n $. On the other hand, let $ A(h,g)\in\tilde{\cR}^n(\K) $. We write $ h=1+h_1x+\ldots $ and $ g=x+g_1x^2+\ldots $ and see that $ 0=[A(h,g)]_{i0}=h_i $ for $ 0<i<n $ so $ h\in \cH^n(\K) $. We also see that $ 0=[A(h,g)]_{i1}=\sum_{k=0}^{i-1}h_kg_{i-1-k}(h_0=g_0=1) $ for $ 1<i<n $. Using induction on $ i $ we get that $ g\in \cN^n(\K) $. Hence, $ A(\cH^n(\K)\rtimes \cN^n(\K))=\tilde{\cR}^n(\K) $.
    \end{proof}
    As every Riordan matrix is a lower triangular matrix, for every $m\in \N$, we have a natural homomorphism $\pi_m:\tilde{\cR}(\K)\to GL_{m}(\K)$ defined by $\pi_{m}((d_{i,j})_{i,j\in \N}):=(d_{i,j})_{i,j=0,\dots, m-1}$ for every $(d_{i,j})_{i,j\in \N}\in \tilde{\cR}(\K)$. For every $m\in \N$, put $\tilde{\cR}_m:=\pi_m(\tilde{\cR}(\K))$ then we have $ \tilde{\cR}_m\cong\tilde{\cR}(\K)/\tilde{\cR}^m(\K) $ and the natural map $\varphi_{m+1}:\tilde{\cR}_{m+1}\to \tilde{\cR}_m$ given by $\varphi_{m+1}(( d_{i,j})_{i,j=0,\dots, m+1}):=(d_{i,j})_{i,j=0,\dots, m})$ is a homomorphism. Then the Riordan group is isomorphic to the inverse limit  group $\varprojlim\{(\tilde{\cR}_m,\varphi_m)_{m\in \N})\}$.

    \begin{lem}
        \label{L-Hm}
        We have that
        \begin{align*}
            \cH^m(\K)=\langle \varphi(g)(h)h^{-1},g\in \cN^{m-1}(\K),h\in \cH(\K) \rangle=\langle \varphi(g)(h)h^{-1},g\in \cN^{m-1}(\K),h=1+x \rangle,
        \end{align*} for any $ m\geq2 $.
    \end{lem}
    \begin{proof}
        Let  $ h=1+a_1x+a_2x^2+\dots\in \cH(\K) $ and  $ g=x+b_mx^m+\dots\in \cN^{m-1}(\K) $.
        From Lemma \ref{L-simple computations} we get that $ \varphi(g)(h)h^{-1}=1+a_1b_mx^m+(a_1b_{m+1}+r_{m+1})x^{m+1}+\dots $, where $ r_n(n> m) $ is a polynomial which depends on $ a_1,\ldots,a_{n-1},b_m,\dots,b_{n-1} $, for every $n\geq m+1$. Hence

        $$\langle \varphi(g)(h)h^{-1},g\in \cN^{m-1}(\K),h\in \cH(\K) \rangle \subset \cH^{m}(\K).$$
        If $h=1+x$ then $h^{-1}=1+\sum_{i=1}^\infty (-1)^ix^i$. Therefore, for every $g=x+\sum_{j=m}^\infty b_jx^j\in \cN^{m-1}(\K)$, we have
        \begin{align*}
            h(g)h^{-1}&=[(1+x)+\sum_{j=m}^\infty b_jx^j](1+\sum_{i=1}^\infty (-1)^ix^i)\\
            &=1+b_mx^{m}+\sum_{j\geq m+1}(\sum_{i=0}^{j-1}(-1)^ib_{j-i})x^j.
        \end{align*}
        By the freedom of choice of $b_j\in \K$ for every $j\geq m$ we get that $$\cH^{m}(\K)\subset\langle \varphi(g)(h)h^{-1},g\in \cN^{m-1}(\K),h=1+x \rangle.$$
    \end{proof}

    Next we will present topological finite generating sets of $\cR(\F_p)$ and $\cR(\Z)$. Let us recall the definition of (topological) generating sets of a profinite group.
    \begin{definition}
        Let $S$ be a subset of a profinite group $G$. We call that $S$ \textit{(topologically) generates} $G$ or $S$ is a \textit{(topologically) generating} set of $G$ if the subgroup $\langle S\rangle$ of $G$ generated by $S$ is dense in $G$. A profinite group $G$ is called finitely generated if it has a finitely generating set $S$.
    \end{definition}

    \begin{proposition}
    \label{P-finitely generated}
        The Riordan groups $ \cR(\F_p) $  and $ \cR(\Z) $ are finitely generated, for every prime number $p\geq 2 $. Furthermore, for every  generating set $S_N$ of the Nottingham group $ \cN(\F_p)$ (respectively $\cN(\Z))$, the group $ \cR(\F_p) $ (respectively $ \cR(\Z) $) is generated by $$ \{(1+x,x), (1,g):g\in S_N\} .$$

    \end{proposition}
    \begin{proof}
        We know that $\cN(\F_p)=\varprojlim \cN(\F_p)/\cN^n(\F_p)$, $\cN(\Z)=\varprojlim \cN(\Z)/\cN^n(\Z)$, and both $\cN(\F_p)$ and $\cN(\Z)$ are topologically finitely generated, see for example \cite[Theorem 4.5]{Babenko13} and \cite[Theorem 1.1]{BB08}. Let $ \K=\F_p $ or $ \Z $ and let $ S_N $ be a finite generator set of $ \cN(\K) $. From Lemma \ref{L-Hm}, the group $ \cH^2(\K)\rtimes\{x\} $ is generated by $ (1+x,x) $ and $ \{1\}\rtimes \cN(\K) $ so for every $m\geq 2$ the group $ \cH^2(\K)\rtimes \cN(\K)/\cR^m(\K) $ is generated by the left cosets of $\cR^m(\K)$: $ (1+x,x)\cR^m(\K) $ and $ \{1\}\rtimes \cN(\K)/\cR^m(\K) $. Therefore the group $ \cH^2(\K)\rtimes \cN(\K)/\cR^m(\K) $ is generated by $ (1+x,x)\cR^m(\K) $ and $ \{(1,s)\cR^m(\K)|s\in S_N\} $ for every $ m\geq 2 $.

        On the other hand, for every $a_1\in \K$, $ 1+a_1x $ and $ \cH^2(\K) $ generate the set $ \{1\pm a_1x+a_2x^2+a_3x^3+\ldots|a_i\in\K,i\geq2 \} $ because of the two following equations. The first one is \[ 1+a_1x+a_2x^2+a_3x^3+\ldots=(1+a_1x)(1+p_2x^2+p_3x^3\ldots), \] where $ p_2=a_2,p_{n+1}=a_{n+1}-a_1p_n $; and the second one is \begin{align*}
            1-a_1x=(1+a_1x+a_1^2x^2+a_1^3x^3+\ldots)^{-1}.
        \end{align*}

        Using induction, we claim that $ 1+x $ and $ \cH^2(\K) $ generate $ 1+a_1x $ for every $a_1\in \K$. In fact, we have the following equation:\[ 1+(a_1+1)x=(1+x)(1+a_1x-a_1x^2+a_1x^3-a_1x^4\ldots). \]
Hence if $\K=\F_p$ we get the claim. If $\K=\Z$ then we get that $ 1+x $ and $ \cH^2(\K) $ generate $ 1+a_1x $ for every $a_1\geq 0$. On the other hand, from above, we also have that $ 1+a_1x $ and $ \cH^2(\K) $ generate $1-a_1x$ for every $a_1\in \K$. Therefore, we also obtain the claim for the case $\K=\Z$.

        Therefore, $ 1+x $ and $ \cH^2(\K) $ generate the whole $ \cH(\K) $. This leads that $ (1+x,x)\cR^m(\K) $ and $ \cH^2(\K)\rtimes \cN(\K)/\cR^m(\K) $ generate the whole group $ \cR(\K)/\cR^m(\K) $ for every $ m\geq 2 $. Therefore, $ (1+x,x)\cR^m(\K) $ and $ \{(1,s)\cR^m(\K)|s\in S_N\} $ can generate the whole $ \cR(\K)/\cR^m(\K) $ for every $ m\geq 2 $. Therefore, we get that $ (1+x,x) $ and $ S_N $ topologically generate $ \cR(\K) $ when $ \K=\F_p $ or $ \K=\Z $.
    \end{proof}

    \subsection{Lower central series of $\cR(\K)$}

    Let $G$ be a group and let $G_1,G_2$ be subgroups of $G$. For elements $g$ and $h$ of $G$, the \textit{commutator} of $g$ and $h$ is $[g,h]:=ghg^{-1}h^{-1}$. The \textit{commutator subgroup} of $G_1$ and $G_2$ is the group generated by $$\{[g_1,g_2]:g_1\in G_1, g_2\in G_2\}.$$ We denote by $\gamma_2(G)=[G,G]$, the commutator subgroup of $G$. The \textit{lower central series} of $G$ is defined inductively as $\gamma_n(G):=[G,\gamma_{n-1}(G)]$ for every $n\geq 3$.

    In this subsection we investigate the lower central series of $\cR(\K)$.

    The following lemma is a slight generalization of \cite[Proposition 3.3]{GG} when the group $H$ is commutative.
    \begin{lem}
        \label{L-commutator subgroups of semi-direct products}
        Let $ G $ be a group and $ G_1,G_2 $ be normal subgroups of $G$. Let $ H $ be a commutative group and $ H_1,H_2 $ be normal subgroups of $H$. Suppose that there is an action $ \varphi:G\to\Aut(H) $ such that we have $ \varphi(g_i)|_{H_i}\in \Aut(H_i) $ for all $ i=1,2 $ and $ g_i\in G_i $. Let $L$ be the subgroup of $H$ generated by $$ \{\varphi(g_1)(h_2)h_2^{-1},\varphi(g_2)(h_1)h_1^{-1}:g_i\in G_i,h_i\in H_i, i=1,2 \}.$$ Then $\varphi(g)|_L\in \Aut(L)$ for every $g\in [G_1,G_2] $ and \begin{align}\label{semi direct prod and commutator}
            [H_1\rtimes_\varphi G_1,H_2\rtimes_\varphi G_2]=L\rtimes_\varphi[G_1,G_2].
        \end{align}
    \end{lem}
    \begin{proof}

        For any $ g,g'\in G $ and $ h\in H $ we have
        \begin{align*}
            \varphi(g)(\varphi(g')(h)h^{-1})=\varphi(gg')(h)h^{-1}\cdot (\varphi(g)(h)h^{-1})^{-1}.
        \end{align*}
        As $G_1$ is a normal subgroup of $G$ we get that $gg'\in G_1$ for every $g\in [G_1,G_2],g'\in G_1 $. Therefore, $\varphi(g)(\varphi(g')(h)h^{-1})=\varphi(gg')(h)h^{-1}\cdot (\varphi(g)(h)h^{-1})^{-1}\in L$ for every $g\in [G_1,G_2],g'\in G_1,h\in H_2 $.   Similarly, $\varphi(g)(\varphi(g')(h)h^{-1})\in L$ for every $g\in [G_1,G_2],g'\in G_2,h\in H_1 $. Hence $\varphi(g)(L)\subset  L$ for every $g\in [G_1,G_2] $.

        Let $g\in [G_1,G_2]$. It is clear that $\varphi(g)$ is injective. The surjectivity of $\varphi(g)|_L$, follows from the following equation.
        \begin{align*}
            \varphi(g)(\varphi(g^{-1}g')(h)h^{-1}\cdot (\varphi(g^{-1})(h^{-1})h))=\varphi(g')(h)h^{-1}, \mbox{ for every } g'\in G,h,h'\in H.
        \end{align*}
        Hence $\varphi(g)|_L\in \Aut(L)$ for every $g\in [G_1,G_2] $.

        Now we are ready to prove \eqref{semi direct prod and commutator}. For every $(h_i,g_i)\in H_i\rtimes_\varphi G_i $ we have
        \begin{align*}
            [(h_1,g_1),(h_2,g_2)]=(\varphi(g_1)(h_2)h_2^{-1}\cdot\varphi(g_1g_2g_1^{-1})(h_1^{-1})h_1\cdot \varphi([g_1,g_2])(h_2^{-1})h_2,[g_1,g_2]).
        \end{align*}
        Therefore $[H_1\rtimes_\varphi G_1,H_2\rtimes_\varphi G_2]\subset L\rtimes_\varphi[G_1,G_2]$.
        On the other hand for every $g_i\in G_i, h_i\in H_i, i=1,2$, we have \begin{align*}
            (\varphi(g_1)(h_2)h_2^{-1},1)=[(1,g_1),(h_2,1)] \text{ and }(\varphi(g_2)(h_1)h_1^{-1},1)=[(h_1,1),(1,g_2)]^{-1}.
        \end{align*}
        Hence, $L\rtimes_\varphi[G_1,G_2]\subset [H_1\rtimes_\varphi G_1,H_2\rtimes_\varphi G_2]$.
    \end{proof}

    Now we are ready to calculate the lower central series of the Riordan group $\cR(\K)$.

    \begin{lem} Let $ \F_q $ be finite field with characteristic $ p>2. $ We have for $ n\geq 2 $,
        \begin{align}
            \label{F-lower central series}
            \gamma_n(\cR(\F_q))=\cH^{n+[(n-2)/(p-1)]}(\F_q)\rtimes \cN^{n+1+[(n-2)/(p-1)]}(\F_q).
        \end{align}
    \end{lem}
    \begin{proof} For convenience, we denote $ n+[(n-2)/(p-1)] $ just by $ \tau_n $.

        Applying Lemmas \ref{L-commutator subgroups of semi-direct products}, \ref{L-Hm} for $ m=2 $, and \cite[Remark 1 ii)]{Camina} (or \cite[Proposition 12.4.24]{LM}), formula \eqref{F-lower central series} is true for the case $ n=2.$

        Suppose that formula \eqref{F-lower central series} is true for $n\geq 2, $ applying Lemma \ref{L-commutator subgroups of semi-direct products} and \cite[Remark 1 ii)]{Camina} (or \cite[Proposition 12.4.24]{LM}) we get \begin{align*}
            \gamma_{n+1}(\cR(\F_q))=&[\cH(\F_q)\rtimes \cN(\F_q),\cH^{\tau_n}(\F_q)\rtimes \cN^{\tau_n+1}(\F_q)]\\
            =&L\rtimes \cN^{\tau_{n+1}+1},
        \end{align*}
        where $ L $ is the subgroup of $ \cH(\F_q) $ generated by $$\{ h'(g)h'^{-1},h(g')h^{-1}:g\in \cN(\F_q),g'\in \cN^{\tau_n+1}(\F_q),h\in \cH(\F_q),h'\in \cH^{\tau_n}(\F_q)\}.$$

        We need to check that $ L=\cH^{\tau_{n+1}}(\F_q)$. Let $g\in \cN(\F_q),g'\in \cN^{\tau_n+1}(\F_q),h\in \cH(\F_q),h'\in \cH^{\tau_n}(\F_q)$ then write \begin{align*}
            h'=&1+a_{\tau_n}x^{\tau_n}+a_{\tau_n+1}x^{\tau_n+1}+\ldots,\\
            g=&x+b_2x^2+b_3x^3+\ldots.\\
        \end{align*}
        Applying Lemma \ref{L-simple computations}, the smallest positive degree for $ h'(g)h'^{-1} $ is $ \tau_n+1 $ and the coefficient of $ \tau_n+i-1 $-degree($ i\geq2 $) term is of the form \[ \tau_na_{\tau_n}b_i+r_{\tau_n+i-1}, \] where $ r_{\tau_n+i-1} $ is a polynomial depends on $ b_2,\ldots,b_{i-1},a_{\tau_n},\ldots,a_{\tau_n+i-1} $. Applying Lemma \ref{L-Hm} for $ m=\tau_n+2 $, $ \cH^{\tau_n+2}(\F_q) =\langle h(g')h^{-1}\rangle $ so $ \cH^{\tau_n+2}(\F_q)\subseteq L\subseteq \cH^{\tau_n+1}(\F_q) $.

        If $ \tau_n $ is not divisible by $p$ then $ \tau_n+1=\tau_{n+1} $, because $ [(n-1)/(p-1)]=[(n-2)/(p-1)] $. Also, because $ \tau_n $ is invertible in $ \F_q $, by choosing $ a'_{\tau_n}=1 $ we are done.

        If $ \tau_n $ is divisible by $p$ then $ \tau_na_{\tau_n}b_i=0 $ in $ \F_q $ so by Lemma \ref{L-simple computations} the coefficient of $ \tau_n+1 $-degree is zero. The smallest positive degree of $ h'(g)h'^{-1} $ is $ \tau_n+2=\tau_{n+1}. $ Hence, $ L= \cH^{\tau_n+2}(\F_q)=\cH^{\tau_{n+1}}(\F_q). $
    \end{proof}
    \begin{definition}
    (\cite{KLP} or \cite[Definition 12.1.5]{LM})
        Let $G$ be an infinite pro-$p$-group. We say that $G$ has finite width if
        $$\sup_{n}\log_p(|\gamma_n(G)/\gamma_{n+1}(G)|)<\infty.$$
    \end{definition}
    The concept of coclass is very useful in approaching a classification of finite p-groups. The notion finite width arises naturally as a generalization of coclass to classify pro-$p$-groups \cite{KLP}, \cite[Chapter 12]{LM}.
    \begin{cor}
    \label{C-finite width}
        Let $\F_q$ be a finite field with characteristic $p$. Then $\cR(\F_q)$ is a pro-$p$-group with finite width.
    \end{cor}
    \begin{proof}
        We see that when we increase $ n $ by $ 1 $ then the smallest positive degree in $ \cH^{n+[(n-2)/(p-1)]} $ and the smallest degree which is greater than $ 1 $ in $ \cN^{n+1+[(n-2)/(p-1)]} $ are increasing by $ 1 $ or $ 2. $ So there are maximum $ q^4 $ elements that are in $ \gamma_n(\cR(\F_q)) $ but not in $ \gamma_{n+1}(\cR(\F_q)) $. Hence, $ |\gamma_n(\cR(\F_q))/\gamma_{n+1}(\cR(\F_q))|\leq q^4 $, this means $ \cR(\F_q) $ has finite width.
    \end{proof}
\begin{proof}[Proof of Theorem \ref{T-main1}]
Theorem \ref{T-main1} follows from Corollaries \ref{C-profinite}, \ref{C-finite width} and Proposition \ref{P-finitely generated}.
\end{proof}



    \section{index-subgroups and Hausdorff dimensions}

    We begin this section by recalling the definitions of Hausdorff dimension, and lower and upper box dimensions in metric spaces. For more details, see \cite{Falconer}.

    Let $(X,d)$ be a metric space, $S$ be a subset of $X$, and $\varepsilon,\alpha>0$. We define
    $$H^\alpha_\varepsilon(S):=\inf\sum_{k} (\diam S_k)^\alpha,$$
    where the infimum is taken over all covers $\{S_k\}_{k\in \N}$ of $S$ with $\diam S_k\leq \varepsilon$ for every $k\in \N$. Here $\diam S_k:=\sup_{x,y\in S_k}d(x,y)$, the diameter of $S_k$. For every $\alpha>0$, $H^\alpha_\varepsilon(S)$ is non-increasing with $\varepsilon$ and hence the limit
    $$H^\alpha(S):=\lim_{\varepsilon\to 0}H^\alpha_\varepsilon(S)$$
    exists.

    The \textit{Hausdorff dimension} of $S$ is defined as follows.
    $$\dim_H(S):=\sup\{\alpha>0:H^\alpha(S)=\infty\}=\inf\{\alpha>0:H^\alpha(S)=0\}.$$
    It is clear that if $S\subset S'$ then $\dim_H(S)\leq \dim_H(S')$.

    Now we assume further that the metric space $(X,d)$ is compact. As $X$ is compact, for every $\varepsilon>0$ and for every cover $\{B(x,\varepsilon)\}_{x\in I}$ of $S$ we can choose a finite subcover. Here, $B(x,\varepsilon):=\{y\in X:d(y,x)<\varepsilon\}$ is the open ball in $X$ with the center at $x$ and the radius $\varepsilon$. We define $N_\varepsilon(X)$ the minimal number of open balls of radius $\varepsilon$ required to cover $S$. The \textit{lower and upper box dimensions} of $S$ are defined as follows.
    $$\underline{\dim}_B(S):=\liminf_{\varepsilon\to 0}\frac{\log N_\varepsilon(S)}{-\log \varepsilon}\mbox{ and   } \overline{\dim}_B(S):=\limsup_{\varepsilon\to 0}\frac{\log N_\varepsilon(S)}{-\log \varepsilon}.$$
    For every bounded subset $S$ of a metric space $X$, we always have $\dim_H(S)\leq \underline{\dim}_B(S)$ \cite[Proposition 3.4]{Falconer}. Note that although \cite[Proposition 3.4]{Falconer} only stated the result for bounded subsets of $X=\R^n$, it still holds for general metric spaces.

    Now let $G$ be a profinite group. A \textit{filtration} is a descending chain of open normal subgroups $G_1=G\unrhd G_2\unrhd\cdots $
    which form a base for the neighborhoods of the identity of $G$. If $\{G_n\}$ is a filtration of $G$ then  $\bigcap_{n\in\N} G_n=\{1_G\}$.

    Given a filtration $\{G_n\}$ of $G$, we define an invariant metric $d$ on $G$ by
    $$d(g,h):=\inf\bigg\{\frac{1}{|G:G_n|}:gh^{-1}\in G_n\bigg\}.$$
    Then every set of diameter $\frac{1}{|G:G_n|}$ is contained in some coset of $G_n$. Let $H$ be a subgroup of $G$. If $\rho=\frac{1}{2|G:G_n|}$ then $N_\rho(H)=|HG_n:G_n|=|H:H\cap G_n|$ and hence
    $$\underline{\dim}_B(H)=\liminf_{n\to \infty}\frac{\log |HG_n/G_n|}{\log |G/G_n|} \mbox{ and   } \overline{\dim}_B(H)=\limsup_{n\to \infty}\frac{\log |HG_n/G_n|}{\log |G/G_n|}.$$

    In \cite[Proposition 2.6]{Ab}, Abercrombie showed that $\dim_H(H)\geq \liminf_{n\to \infty}\frac{\log |HG_n/G_n|}{\log |G/G_n|}$ for every closed subgroup $H$ of $G$. Therefore, for every closed subgroup $H$ of $G$ and every number $p\geq 2$, we have that \cite[Theorem 2.4]{BS}
    $$\dim_H(H)=\underline{\dim}_B(H)=\liminf_{n\to \infty}\frac{\log |HG_n/G_n|}{\log |G/G_n|}=\liminf_{n\to \infty}\frac{\log_p |HG_n/G_n|}{\log_p |G/G_n|}.$$
    The Hausdorff dimension does depend on the choice of the filtration $\{G_n\}$ \cite[Example 2.5]{BS}.

    The \textit{Hausdorff spectrum} of the profinite group $G$ is the set
    $$\hspec(G):=\{\dim_H(H):H \mbox{ is a closed subgroup of } G\}.$$
    The study of Hausdorff spectrum of profinite groups has been initiated by Barnea and Shalev \cite{BS}.
    In \cite[Theorem 1.6]{BS} we know that for Nottingham group $\cN(\F_p)$ with prime $p\geq 5$ it holds that for every closed subgroup $H$ of $\cN(\F_p)$, $\dim_H(H)$ does not depend on the choice of the filtration and
    $$\{0\}\bigcup\bigg\{\frac{1}{n}:n\in\N\bigg\}\subset \rm{hspec}(\cN(\F_p))\subset [0,\frac{3}{p}]\bigcup\bigg\{\frac{1}{n}:n\in \N\bigg\}.$$
    Furthermore, later via introducing index-subgroups of $\cN(\F_p)$, Barnea and Klopsch showed that \cite[Theorem 1.8]{Index} for every $p>2$,
    $$\bigg\{\frac{1}{n}:n\in\N\bigg\}\cup [0,\frac{1}{p}]\cup\bigg\{\frac{1}{p}+\frac{1}{p^r}:r\in \N\bigg\}\subset \rm{hspec}(\cN(\F_p)).$$

    Inspired by \cite{Index}, in this section we will introduce index-subgroups of $\cR(\F_p)$ and study $\hspec(\cR(\F_p))$. We start with constructions of some filtrations of $\cR(\K)$.
    \begin{lem}\label{Gsigma}
        Let $\K$ be a unital commutative ring. Let $ \sigma:\N\to\N $ be a function such that $ \sigma(1)=1,\sigma(m+n)\leq \sigma(m)+\sigma(n) $, non-decreasing and $ \lim_{n\to\infty}\sigma(n)=\infty $. Then $ G_n^{\sigma}=\cH^{\sigma(n)}(\K)\rtimes \cN^n(\K) $ is a filtration of $\cR(\K)$ satisfying that $[G^\sigma_i,G^\sigma_j]\subset G^\sigma_{i+j}$ for every $i,j\in \N$.
    \end{lem}
    \begin{proof}
        From Lemma \ref{Rmn}, we see that $ G^\sigma_n $ form a chain of descending chain of normal subgroups of $\cR(\K)$.
        Remember that for $ h=1+ax^n+\mathscr{O}(x^{n+1})\in \cH(\K) $ and $ g=x+bx^{m+1}+\mathscr{O}(x^{m+2})\in \cN(\K) $, we have \begin{align*}
            h(g)h^{-1}=(h+abnx^{n+m}+\mathscr{O}(x^{n+m+1}))h^{-1}\in \cH^{m+n}(\K) .
        \end{align*}
        Combining with the following equation
          \begin{align*}
            [(h_1,g_1),(h_2,g_2)]=(\varphi(g_1)(h_2)h_2^{-1}\cdot\varphi(g_1g_2{\bar g_1})(h_1^{-1})h_1\cdot \varphi([g_1,g_2])(h_2^{-1})h_2,[g_1,g_2]),
        \end{align*}
        we get that $$ [\cH^{\sigma(m)}(\K)\rtimes \cN^m(\K),\cH^{\sigma(n)}(\K)\rtimes \cN^n(\K)]\subseteq \cH^{\min\{\sigma(m)+n,m+\sigma(n)\}}(\K)\rtimes \cN^{m+n}(\K) .$$
        By the properties of $ \sigma $, we have $ \sigma(m)+n\geq\sigma(m)+\sigma(n)\geq\sigma(m+n) $ and similarly for $ m+\sigma(n)\geq\sigma(m+n) $. Hence, $ G_n^\sigma $ is a filtration with the property we are looking for.
    \end{proof}
    Now we introduce the important notion of this section, the index-subgroups of $\cR(\F_p)$.
    For every $ I,J\subseteq\N $ we define
    \begin{align*}
        \cH(I):&=\{1+\sum_{i\in I}a_ix^i|a_i\in\F_p\},\\
        \cN(J):&=\{x+\sum_{j\in J}b_jx^{j+1}|b_j\in\F_p\},\\
        \cR(I,J):&=\{(1+\sum_{i\in I}a_ix^i,x+\sum_{j\in J}b_jx^{j+1})|a_i,b_j\in\F_p\}.
    \end{align*}

    It is clear the $\cR(I,J)$ is a closed subset of $\cR(\F_p)$ for every $ I,J\subseteq\N $. We call $ (I,J) $ an \textit{admissible index-pair} if $ \cR(I,J) $ is a subgroup of $\cR(\F_p)$. If $(I,J)$ is an admissible index-pair then we say that $\cR(I,J)$ is an \textit{index-subgroup} of $\cR(\F_p)$.
    \begin{lem}
        \label{H(I)}
        The set $ \cH(I) $ is a subgroup of $ \cH(\F_p) $ if and only if $ I $ is closed under addition.
    \end{lem}
    \begin{proof}
        From the follow equation \begin{align*}
            (1+\sum_{i\in I}a_ix^i)(1+\sum_{i'\in I}a'_{i'}x^{i'})=&1+\sum_{i,i'\in I}a_ia'_{i'}x^{i+i'},\\
        \end{align*}$ \cH(I) $ is closed under multiplication if and only if $ I $ is closed under addition. We only need to prove that $ \cH(I) $ is closed under inversion if $ I $ is closed under addition. Suppose there exists $ h\in \cH(I) $ such that $ h^{-1}\notin \cH(I) $. Let $ i_0 $ be the smallest positive degree of $ h^{-1} $ with non-zero coefficient such that $ i_0\notin I $. The coefficient of $ i_0 $-degree in $ hh^{-1} $ is $ c=\sum_{i=0}^{i_0}a_ia'_{i_0-i} $, where $ a_i,a'_j $ are coefficients at degree $ i,j $ of $ h,h^{-1} $ respectively. As $ i_0\notin I $, we see that $ a_ia'_{i_0-i}\neq0 $ only if $ i=0 $ or $ i,i_0-i\in I $. But $ i,i_0-i $ cannot both be inside $ I $ when $ I $ is closed under addition so the only case $ a_ia'_{i_0-i}\neq0 $ is when $ i=0 $. Hence, $c= a'_{i_0}\neq0 $, which is a contradiction.
    \end{proof}
    \begin{lem}\label{admissible pair} Let $I,J\subset \N$. Then
        $ (I,J) $ is an admissible index-pair if and only if all the followings happen\begin{enumerate}
            \item For $ j\in J $ and $ n\in\{1,\ldots j+1\} $ with $ \left(\begin{array}{c}
                j+1\\n
            \end{array}\right) $ is not divisible by $ p $ we have $ \{j+nj'|j'\in J\}\subseteq J; $
            \item $ I $ is closed under addition;
            \item For $ i\in I $ and $ n\in\{1,\ldots i\} $ with $ \left(\begin{array}{c}
                i\\n
            \end{array}\right) $ is not divisible by $ p $ we have $ \{i+nj|j\in J\}\subseteq I. $
        \end{enumerate}
    \end{lem}
    \begin{proof}
        First, we show the ``only if" part. Assume that $ (I,J) $ is an admissible index-pair. As $ \cN(J) $ is a subgroup of $ N(\F_p) $ applying \cite[Theorem 1.1]{Index}, we get condition (1). On the other hand, $ \cH(I) $ is a subgroup of $ \cH(\F_p) $, applying Lemma \ref{H(I)} we get (2).

        For the last condition, we observe that for every $ h\in \cH(I) $ and $ g\in \cN(J) $ we have $ (1,g)(h,{\bar g})=(h(g),x)$ and hence $h(g)\in \cH(I) $ .
        Take $ h=1+x^i $ and $ g=x+x^{j+1} $ with $i\in I$ and $j\in J$, we have \begin{align*}
            h(g)=1+(x+x^{j+1})^i=1+\sum_{n=0}^{i}\left(\begin{array}{c}
                i\\n
            \end{array}\right)x^{nj+i}.
        \end{align*} Therefore, if $ \left(\begin{array}{c}
            i\\n
        \end{array}\right) $ is not divisible by $ p $ we have $ \{i+nj|j\in J\}\subseteq I. $

        Now, we show the ``if" part. From the definition of the operation on Riordan group
              \begin{align*}
		(h',g)(h,g')=(h'h(g),g'(g))\text{ and }(h,g)^{-1}=(h^{-1}({\bar g}),{\bar g}).
        \end{align*}

        To make $ \cR(I,J) $ a group we need $ h'h(g),h^{-1}({\bar g})\in \cH(I) $ and $ g'(g),{\bar g}\in \cN(J) $ for every $ h,h'\in \cH(I) $ and $ g,g'\in \cN(J).$
        From (1) and \cite[Theorem 1.1]{Index} we get that $ g'(g),{\bar g}\in \cN(J) $ for every $g,g'\in \cN(J)$. Lemma \ref{H(I)} and condition (2) give us that $ \cH(I) $ is a subgroup of $ \cH(\F_p) $. Thus, we now only need to show $ h(g)\in \cH(I)$ for every $h\in \cH(I),g\in \cN(J)$. Put
        $$ \mathscr{S}:= \{x+ax^{j+1}:j\in J, a\in \F_p\}.  $$
        Let  $ \widehat{\mathscr{S}} $  be the sub-semigroup generated by $\mathscr{S}$, i.e.
       $$ \widehat{\mathscr{S}}:=\{g_1\circ g_2\circ\cdots \circ g_s|s\in \N, g_i\in  \mathscr{S} \mbox{ for every } 1\leq
      s\}.$$

        We endow $\cN(\F_p)=\varprojlim \cN(\F_p)/\cN^{n+1}(\F_p)$ with its inverse limit topology. Then $\operatorname{cl}(h(\widehat{\mathscr{S}}))$, the topological closure of $h(\widehat{\mathscr{S}})$ in $\cN(\F_p)$ is also a semigroup and hence from \cite[Lemma 2.2]{Index} it is a subgroup of $\cN(\F_p)$. For $ h=1+\sum_{i\in I}a_ix^i\in \cH(I) $ and $ g=x+bx^{j+1}\in\mathscr{S} $, we have \begin{align*}
            h(g)=&1+\sum_{i\in I}\sum_{n=0}^{i}a_i\left(\begin{array}{c}
                i\\n
            \end{array}\right)b^nx^{nj+i}\in \cH(I).
        \end{align*} So $ h(g)\in \cH(I) $ for any $ h\in \cH(I) $ and $ g\in\widehat{\mathscr{S}}. $ From continuity, $ h(\operatorname{cl}(\widehat{\mathscr{S}}))\subseteq\operatorname{cl}(h(\widehat{\mathscr{S}}))\subseteq \cH(I). $ The result follows.
    \end{proof}
    From now on, for simplicity, we denote $\cN=\cN(\F_p), \cN^n=\cN^n(\F_p), \cH^n=\cH^n(\F_p)$ and $\cR=\cR(\F_p)$ for every $n\in \N$.
    We also fix the filtration $\{ \cH^n\rtimes \cN^n \}_n$. Before studying the Hausdorff dimensions of index-subgroups of $\cR$, let us recall the lower and upper density. Let $I\subset \N$. The \textit{lower and upper density} of $I$ are defined as follows.
    \begin{align*}
        \ldense(I):&=\liminf_{m\in \N}\frac{|\{i\in I:i\leq m\}|}{m},\\
        \udense(I):&=\limsup_{m\in \N}\frac{|\{i\in I:i\leq m\}|}{m}.
    \end{align*}
    If $\ldense(I)=\udense(I)$ then the \textit{density} of $I$, $\dense(I)$, is defined as this common value.

    We define \begin{align*}
        \inspec(\cR):=\{\dim_H(\cR(I,J))|(I,J)\text{ is an admissible index-pair}\}\subset \hspec(\cR).
    \end{align*}

    Before investigating further $\dim_H(\cR(I,J))$, where $(I,J)$ is an admissible index-pair, let us define functions $W:\N\to [0,1]$ and $w:(p\N-1)\to [0,\dfrac{1}{p}]$, which have been used in \cite{Index}, as follows. For $m\in \N$ with $m=\sum_{n=0}^\infty m_np^n$, we put
    $$W(m):=\sum_{n=0}^\infty m_np^{-n-1}, \mbox{ } m_n\in \{0,1,\dots, p-1\}.$$
    The function $w: p\N-1\to [0,\dfrac{1}{p}]$ is defined by $w(j):=W(j+1)$ for every $j\in p\N-1$.
    For every $\xi\geq 0$, we define $$J(\xi):=\{j\in (p\N-1)|w(j)<\xi\}.$$

    The following lemma is an adaptation of \cite[Lemma 4.1]{Index}.
    \begin{lem}
        \label{L-the reverse inclusion}
        Let $ \xi\in[0,\dfrac{1}{p}] $ then $J(\xi)$ satisfies the condition 1) of Lemma \ref{admissible pair} and $ \dense(J(\xi)\cap s\N)=s^{-1}\xi $ for any $ s $ which is not divisible by $ p. $
    \end{lem}
    \begin{proof}
        From \cite[Example 3.5]{Index} we know that $J(\xi)$ satisfies the condition 1) of Lemma \ref{admissible pair}.

        We write $\xi$ in the form $\sum_{k=1}^\infty \xi_kp^{-k-1}$, where $\xi_k\in \{0,1,\dots, p-1\}$. For every $m\in \N$, we define
        \begin{align*}
            J_m:&=\bigcup_{0\leq t<\xi_m}\{i\in s\N|i\equiv_{p^{m+1}}-1+\sum_{n=1}^{m-1}\xi_np^n+tp^m\}, \mbox{ and }\\
            S_m:&=\bigcup_{\xi_m< t\leq p-1}\{i\in s\N|i\equiv_{p^{m+1}}-1+\sum_{n=1}^{m-1}\xi_np^n+tp^m\}.
        \end{align*}
        For every $m\in \N$ and $0\leq t<\xi_m$, because $ s $ is coprime with $ p^{m+1} $, there exists a unique element $ k\in\{1,\ldots,p^{m+1}-1\} $ such that $ ks\equiv_{p^{m+1}}-1+\sum_{n=1}^{m-1}\xi_np^n+tp^m $. Thus, for every $m\in \N$, $ J_m $ is a union of disjoint arithmetic progressions with increment $ sp^{m+1} $.  Note that the density of an arithmetic progression with increment $t$ is $1/t$. Hence $\dense(J_m)=\dfrac{\xi_m}{s}p^{-m-1}$ for every $m\in \N$.
        Therefore, for every $m\in \N$, we have
        \begin{align*}
            \ldense(J(\xi))&\geq \ldense(\bigcup_{i=1}^m J_i)\\
            &=\sum_{i=1}^m\dense(J_i)\\
            &=\frac{1}{s}\sum_{i=1}^{m}\xi_ip^{-i-1}.
        \end{align*}
        Letting $m\to \infty$ we get that $\ldense(J(\xi))\geq s^{-1}\xi$.

        Similarly, for every $m\in \N$ we have
        \begin{align*}
            \ldense(S)&\geq \ldense(\bigcup_{i=1}^m S_i)\\
            &=\sum_{i=1}^m\dense(S_i)\\
            &=\frac{1}{s}\sum_{i=1}^{m}(p-1-\xi_i)p^{-i-1}.
        \end{align*}
        Letting $m\to \infty$ we get that $\ldense(S)\geq s^{-1}(1/p-\xi)$. On the other hand, we have $J(\xi)$ and $S$ are disjoint, $J(\xi)\cup S=s\N\cap (p\N-1)\setminus w^{-1}(\xi)$, and the set $w^{-1}(\xi)$ has at most one element. Hence, as $s$ is coprime with $p$ we get that
        $$\udense(J(\xi))=\frac{1}{s}[1/p-(1/p-\xi)]=\frac{\xi}{s}.$$
    \end{proof}
    Inspired by \cite[Proposition 4.3]{Index}, we characterize the index-subgroups of $\cR(\F_p)$ as follows.
    \begin{lem}
        \label{L-characterizations of admissible pairs}
        Let $(I,J)$ be an admissible pair. Then one of the following holds:
        \begin{enumerate}
            \item $I$ is empty and $J$ satisfies the condition 1) of Lemma \ref{admissible pair};
            \item $I$ is a cofinite subset of $ sp^r\N $, where $ s\in\N $ is not divisible by $ p $ and $ r\in\N\cup\{0\}$. In this case, $J\subseteq s\N$ and furthermore $J$ must satisfies one of the following properties:
            \begin{enumerate}[label=(\roman*)]
                \item $ J\subseteq(p\N-1)\cap s\N $ and so $\ldense(J)\in [0,\dfrac{1}{sp}]$;
                \item there exists $ s_0\in\N $ such that $ J $ is a cofinite set of $ s_0\N $ and so $ \dense(J) $ has the form of $ \dfrac{1}{su} $, for some $u\in \N$;
                \item There exist $ s_1,v\in\N $ with $ s_1 $ is not divisible by $ p $ such that $ J\subseteq s_1\N\cap(p^v\N\cup(p\N-1)) $ and $ J\cap(p^v\N\cup(p^v\N-1)) $ is cofinite subset of $ s_1\N\cap(p^v\N\cup(p^v\N-1)) $. Moreover, there exists $ t\in\{1,\ldots,p^{v-1}\} $, $ u\in\N $ such that $ \dense(J)=\dfrac{1+t}{sup^v} $.
            \end{enumerate}
        \end{enumerate}
    \end{lem}
    \begin{proof}
        As $(I,J)$ is an admissible pair, applying Lemma \ref{admissible pair} we get that $I$ is closed under addition and hence $ I $ is either empty or there exists $ t\in\N $ such that $ I $ is an cofinite subset of $ t\N $. In fact, take $ t=\operatorname{gcd}(I) $ when $ I $ is not empty, we prove that for large enough $ N\in\N $ we have $ \{sn|n\geq N,n\in\N\}\subseteq I $. Because $ I $ is non-empty, take two elements $ a_0, a_1\in I $ which are not necessarily different. Let $ s_1:=\operatorname{gcd}(a_0,a_1) $, then there exist $ u,v\in\N_0 $ such that $ 1=ub_0-vb_1 $, where $ a_i=b_is_1 $. Then $ (kb_1+l)s_1=lua_0+(k-lv)a_1\in I $ for any $ k\geq b_1v,0\leq l\leq b_1-1 $. In other word, $ \{s_1n|n\geq b_1^2v,n\in\N\}\subseteq I $. If $ s_1=t $ then we are done. If $ s_1>t $ then there exists $ a_2\in I $ such that $ t\leq s_2:=\operatorname{gcd}(s_1,a_2)<s_1 $. There exists $ a_1'\in\{s_1n|n\geq b_1^2v,n\in\N\} $ such that $ s_2=\operatorname{gcd}(a_1',a_2) $ (take $ n $ is coprime with $ s_2 $). With the same argument as $ s_1 $ there exists $ N_2 $ such that $ \{s_2n|n\geq N_2,n\in\N\}\subseteq I $. We make the same process until we get $ s_m=t$.

        If $ I $ is empty then we get the case 1).

        Now we consider that $ I$ is a cofinite subset of $q\N$ for some $q\in \N$. We can write $q$ as $ sp^r $, where $ s\in\N $ is not divisible by $ p $ and $ r\in\N\cup\{0\}. $ As $s$ is not divisible by $p$, we get that $ \left(\begin{array}{c}
            sp^r\\p^r
        \end{array}\right) $ is not divisible by $ p $ and hence $ I+p^rJ\subseteq I $. This leads to $ J\subseteq s\N $.

        If $ J\subseteq(p\N-1)\cap s\N $ we get the case i) and $ \ldense(J)\in[0,\dfrac{1}{sp}] $.

        Now assume that $J\nsubseteq (p\N-1)$. We define $$s_0:=\operatorname{gcd}(S), \mbox{ where } S:=J\setminus (p\N-1).$$
        For every $j\in S$ we have $ \left(\begin{array}{c}
            j+1\\1
        \end{array}\right) $ is not divisible by $ p $ and hence from Lemma \ref{admissible pair} we get that $J+S\subset J$. Hence there exists $N\in \N$ such that $J+\{s_0n:n\in \N \mbox{ with } n\geq N\}\subset J$.
        We define $$T:=(J+s_0\Z)\cap \{1,2,\dots, s_0\}.$$ Then $s_0\in T$. There exists a large enough $M\in \N$ such that
        \begin{align}
            \label{F-temp}
            \{j\in J|j>s_0M\}=\{s_0n|n\in \N \mbox{ with } n\geq M\}+T.
        \end{align}
        Then in particular,
        \begin{align}
            \label{F-temp 1}
            \dense(J)=\frac{|T|}{s_0}.
        \end{align}
        If $J\subset s_0\N$ then
        $s_0\N\setminus J$  finite and $\dense(J)=\dfrac{1}{s_0}$. We get the case ii). Since $ J\subseteq s\N $ that only happens if $ s|s_0 $. Then $ \dense(J) $ has the form of $ \dfrac{1}{su} $ for some $u\in \N$.

        Next, let us assume that $J\nsubseteq s_0\N$. Let $j\in J\setminus s_0\N$ then by the definition of $s_0$ we get that $j\in (p\N-1)$ and hence
        $$J\subset s_0\N\cup (p\N-1).$$
        On the other hand, from \eqref{F-temp}, we get $j+\{s_0n|n\geq M\}\subset J\setminus s_0\N\subset (p\N-1)$. Hence $s_0$ is divisible by $p$. Therefore, $s_0=p^vs_1$ for some $v,s_1\in \N$ with $s_1$ is not divisible by $p$.

        Now we will prove that $J\subset s_1\N$. As $S\subset s_0\N\subset s_1\N$, it is sufficient to show that $J\cap (p\N-1)\subset s_1\N$. Let $j\in J\cap (p\N-1)$. Let $p^k$ be the highest power of $p$ dividing $M+1$. Applying \cite[Lemma 2.1]{Index} we have that $(M+1)p^vs_1=(M+1)s_0\in J$ satisfies $ \left(\begin{array}{c}
            (M+1)s_0+1\\p^{v+k}
        \end{array}\right) $ is not divisible by $p$. Therefore applying Lemma \ref{admissible pair}, we get $\ell:=(M+1)p^vs_1+p^{v+k}j\in J$.
        As $\ell+1$ is not divisible by $p$, we have $\ell\in s_1\N$ and hence $p^{v+k}j=\ell-(M+1)p^vs_1\in s_1\N$. Thus, $j\in s_1 \N$ and so we obtain that $J\subset s_1\N$.

        By \eqref{F-temp}, we have $J\subset s_1\N\cap (p^v\N\cup (p\N-1))$ and hence
        $$2\leq |T|\leq 1+|\{n|0<n<s_0 \mbox{ with } s_1|n, p|(n+1)\}|=1+p^{k-1}.$$
        From \eqref{F-lower central series} and $ s|s_1 $, we get that $\dense(J)=\dfrac{|T|}{s_0}=\dfrac{1+t}{sup^v}$ for some $u\in \N, t\in\{1,2,\dots, p^{k-1}\}.$

        Now we show that $J\cap (p^v\N\cup (p^k\N-1))$ is a cofinite subset of $s_1\N\cap (p^v\N\cup (p^k\N-1))$. Since $p^v$ and $s_1$ are coprime, from \eqref{F-temp}, it is sufficient to prove that $J\cap (p^v\N-1)\neq\varnothing$. We will prove this by induction. By assumption, we have $J\cap (p\N-1)\neq \varnothing$. Suppose that $j=np^a-1\in J$ for some $a,n\in \N$ with $n$ is not divisible by $p$. Then from Lemma \ref{admissible pair}, $n^2p^{2a}-1=j+(j+1)j\in J$ and therefore $J\cap (p^{a+1}\N-1)\neq \varnothing$. Hence by induction, $J\cap (p^v\N-1)\neq \varnothing$.
    \end{proof}

    \begin{proof}[Proof of Theorem \ref{T-main 2}]
        Let $(I,J)$ be an admissible pair then
        \begin{align*}
            \dim_H(\cR(I,J))=&\liminf_{n\in \N}\dfrac{\log|\cH(I)\cH^{n+1}/\cH^{n+1}|+\log|\cN(J)\cN^{n+1}/\cN^{n+1}|}{\log|\cH/\cH^{n+1}|+\log|\cN/\cN^{n+1}|}\\=&\liminf_{n\in \N}\dfrac{|\{i\in I|i\leq n \}|+|\{ j\in J|j\leq n \}|}{2n}\\=&\dfrac{1}{2}(\ldense(I)+\ldense(J)).
        \end{align*}
        If $(I,J)$ satisfies the condition 1) in Lemma \ref{L-characterizations of admissible pairs} then combining Lemma \ref{admissible pair} (1) and \cite[Theorems 1.1 and 1.8]{Index} we get that \begin{align}
            \label{F-temp 2}
            [0,\dfrac{1}{2p}]\cup\bigg\{\dfrac{1}{2p}+\dfrac{1}{2p^r}|r\in\N\bigg\}\cup\bigg\{\dfrac{1}{2s}|s\in\N\bigg\}=\{\dim_H(\cR(\emptyset,J))|J\text{ satisfies Lemma \ref{admissible pair} (1)}\}.
        \end{align}

        If $(I,J)$ satisfies the condition 2)i) in Lemma \ref{L-characterizations of admissible pairs} then for this case, if $ r\geq 1 $ $ \dim_H(\cR(I,J))\leq\dfrac{1}{2sp^r}+\dfrac{1}{2sp}\leq\dfrac{1}{p} $. If $ r=0 $, $ \dim_H(\cR(I,J))\leq\dfrac{1}{2s}+\dfrac{1}{2sp} $ and if $ s> p $ we still have $ \dim_H(\cR(I,J))\leq\dfrac{1}{p} $.

        If $(I,J)$ satisfies the condition 2)ii) in Lemma \ref{L-characterizations of admissible pairs} then $$ \dim_H(\cR(I,J))=\dfrac{1}{2sp^r}+\dfrac{1}{2su}. $$

        If $(I,J)$ satisfies the condition 2)ii) in Lemma \ref{L-characterizations of admissible pairs} then  $ \dim_H(\cR(I,J))=\dfrac{1}{2sp^r}+\dfrac{1+t}{2sup^v} $. If $ r>1 $ or $ s>1 $ or $ u>1 $ or $ t<p^{v-1} $ then $ \dim_H(\cR(I,J))\leq\dfrac{1}{p} $. If $ r\leq 1,s=1,u=1 $ and $ t=p^{v-1} $ then $ \dim_H(\cR(I,J))=\dfrac{1}{2p^r}+\dfrac{1}{2p}+\dfrac{1}{2p^v} $.

        Combining all these cases, we have \begin{align*}
            \inspec(\cR)\subseteq[0,\dfrac{1}{p}]&\cup\bigg\{\dfrac{1}{p}+\dfrac{1}{2p^r}|r\in\N\bigg\}\cup\bigg\{\dfrac{1}{2}+\dfrac{1}{2p}+\dfrac{1}{2p^r}|r\in\N\bigg\}\cup\\&\cup\bigcup_{s<p}[\dfrac{1}{2s},\dfrac{1}{2s}(1+\dfrac{1}{p})]\cup\bigg\{\dfrac{1}{2sp^r}+\dfrac{1}{2su}|s,u\in\N,r\in\N\cup\{0\}\bigg\}.
        \end{align*}
        Now we will prove the ``$\supseteq$'' inclusion.

        Let $ I=p\N,J=I(\xi) $ for $ \xi\in[0,1/p] $. Then by Lemmas \ref{admissible pair} and \ref{L-the reverse inclusion} we get that $(I,J)$ is an admissible pair and $ \dim_H(\cR(I,J))=\dfrac{1}{2p}+\dfrac{\xi}{2} $. So $ [\dfrac{1}{2p},\dfrac{1}{p}]\subseteq\inspec(\cR) $.

        Also with $ I=p\N $ we can take $ J=p^r\N\cup(p\N-1) $ and so $ \{\dfrac{1}{p}+\dfrac{1}{2p^r}|r\in\N\}\subseteq\inspec(\cR) $. Take $ I=\N,J=p^r\N\cup(p\N-1) $ we have $ \bigg\{\dfrac{1}{2}+\dfrac{1}{2p}+\dfrac{1}{2p^r}|r\in\N\bigg\}\subseteq\inspec(\cR) $.

        Take $ I=s\N,J=I(\xi)\cap s\N $ for $ s<p,\xi\in[0,1/p] $. By Lemmas \ref{admissible pair} and \ref{L-the reverse inclusion} we have $ \dim_H(\cR(I,J))=\dfrac{1}{2s}+\dfrac{\xi}{2s} $, which means $ [\dfrac{1}{2s},\dfrac{1}{2s}(1+\dfrac{1}{p})]\subseteq\inspec(\cR) $.

        For $ \bigg\{\dfrac{1}{2sp^r}+\dfrac{1}{2su}|s,u\in\N,r\in\N\cup\{0\}\bigg\}\subseteq\inspec(\cR) $, we take $ I=sp^r\N $ and $ J=su\N $.
    \end{proof}

    \begin{remark}

    \begin{itemize}
            \item[{\rm (a)}]  In contrast to the Nottingham group $\cN(\F_p)$, Hausdorff dimensions of the Riordan group $\cR(\F_p)$ do depend on the choice of the filtrations.
            \item[{\rm (b)}] If $ \cR(I,J) $ is just infinite then either $ I $ or $ J $ is an empty set. Recall that an infinite topological group is \textit{just infinite} if every non-trivial closed normal subgroup has finite index.
        \end{itemize}
    \end{remark}
    \begin{proof}
        \begin{itemize}
            \item[{\rm (a)}] Look back to Lemma \ref{Gsigma}, when we choose $ \sigma(n)=n $ we have $$ \dim_H(\cR(I,J))=\frac{1}{2}(\ldense(I)+\ldense(J)) .$$ But if we choose $ \sigma(n)=\lceil n/2\rceil $ we have $$ \dim_H(\cR(I,J))=\frac{1}{3}\ldense(I)+\frac{2}{3}\ldense(J) .$$
            \item[{\rm (b)}] Suppose that $ I $ and $ J $ are not empty. The set $ \cH(I)\rtimes\{x\} $ is a non-trivial closed normal subgroup of $ \cR(I,J) $ with the index equal to $ p^{|J|} $. Because $ \cR(I,J) $ is just infinite so $ |J| $ must be finite. Let $ j=\max J=sp^r-1 $ then $ j+p^rj\in J $, which is a contradiction.
        \end{itemize}
    \end{proof}

    \begin{remark}
    From \cite[Theorem 1.6]{BS} we know that $\hspec(\cN(\F_p))\cap (\dfrac{1}{2},1)=\varnothing$ if $p>5$ and $\hspec(\cN(\F_p))\cap (\dfrac{3}{5},1)=\varnothing$ if $p=5$. In particular, 1 is an isolated point of $\hspec(\cN(\F_p))$ for $p\geq 5$. It is interesting to investigate whether 1 is an isolated point of $\hspec(\cR(\F_p))$.
    \end{remark}
    
\end{document}